\documentclass[11pt]{amsart}
\usepackage{amsmath, amsthm, amsfonts, hyperref, graphicx, ifpdf}
\usepackage{mathrsfs,epsfig}
\usepackage[dvipsnames,usenames]{color}
\hypersetup{
   unicode=false,          % non-Latin characters in Acrobat bookmarks
   pdftoolbar=true,        % show Acrobat toolbar?
   pdfmenubar=true,        % show Acrobat menu?
   pdffitwindow=false,     % window fit to page when opened
   pdfstartview={},    % fits the width of the page to the window
   pdftitle={},    % title
   pdfauthor={},     % author
   pdfsubject={},   % subject of the document
   pdfcreator={},   % creator of the document
   pdfproducer={}, % producer of the document
   pdfkeywords={}, % list of keywords
   pdfnewwindow=true,      % links in new window
   colorlinks=true,       % false: boxed links; true: colored links
   linkcolor=blue,          % color of internal links
   citecolor=blue,        % color of links to bibliography
   filecolor=blue,      % color of file links
   urlcolor=blue          % color of external links
}

\newtheorem{theorem}{Theorem}[section]
\newtheorem{cor}[theorem]{Corollary}
\newtheorem{lem}[theorem]{Lemma}
\newtheorem{pro}[theorem]{Proposition}
\newtheorem{remark}[theorem]{Remark}
\newtheorem{remarks}[theorem]{Remarks}

\newtheorem{Def}[theorem]{Definition}

\newtheorem{conjecture}[theorem]{Conjecture}
\theoremstyle{definition}

\renewcommand{\qed}{\mbox{}\hfill\openbox}

\newcommand{\mfd}{manifold}
\newcommand{\nbhd}{neighbourhood}

\newcommand{\be}{\begin{equation}}
\newcommand{\ene}{\end{equation}}
\newcommand{\br}{\begin{remark}}
\newcommand{\er}{\end{remark}}
\newcommand{\bl}{\begin{lem}}
\newcommand{\el}{\end{lem}}
\newcommand{\bcor}{\begin{cor}}
\newcommand{\ecor}{\end{cor}}
\newcommand{\bpro}{\begin{pro}}
\newcommand{\epro}{\end{pro}}
\newcommand{\ben}{\begin{enumerate}}
\newcommand{\een}{\end{enumerate}}
\newcommand{\bp}{\begin{proof}}
\newcommand{\ep}{\end{proof}}
\newcommand{\bpo}{\begin{pro}}
\newcommand{\epo}{\end{pro}}
\newcommand{\beq}{\begin{equation*}}
\newcommand{\eeq}{\end{equation*}}
\newcommand{\bear}{\begin{eqnarray}}
\newcommand{\eear}{\end{eqnarray}}
\newcommand{\beqar}{\begin{eqnarray*}}
\newcommand{\eeqar}{\end{eqnarray*}}
\newcommand{\brem}{\begin{remark*}}
\newcommand{\erem}{\end{remark*}}
\newcommand{\bt}{\begin{theorem}}
\newcommand{\et}{\end{theorem}}

\newcommand{\R}{\mathbb{R}}
\newcommand{\N}{\mathbb{N}}

\renewcommand{\H}{\mathbb{H}}

\newcommand{\sF}{\mathcal{F}}

\DeclareMathOperator{\Hess}{Hess}%

\DeclareMathOperator{\Hom}{Hom}
\DeclareMathOperator{\Sym}{Sym}

\renewcommand{\qed}{\mbox{}\hfill\openbox}
\numberwithin{equation}{section}

\allowdisplaybreaks

\def\XXint#1#2#3{{\setbox0=\hbox{$#1{#2#3}{\int}$}
    \vcenter{\hbox{$#2#3$}}\kern-.5\wd0}}

\makeatletter
\def\@citestyle{\m@th\upshape\mdseries}
\def\citeform#1{{\bfseries#1}}
\def\@cite#1#2{{%
  \@citestyle[\citeform{#1}\if@tempswa, #2\fi]}}
\@ifundefined{cite }{%
  \expandafter\let\csname cite \endcsname\cite
  \edef\cite{\@nx\protect\@xp\@nx\csname cite \endcsname}%
}{}
\makeatother
\begin{document}
\title[No geometric minimal foliations]{Non-Existence of Geometric Minimal Foliations in Hyperbolic Three-Manifolds}
%\title[Riemann Surfaces]{Riemann Surfaces}

%------------------------------------------------------

\author{Michael Wolf}
\address[M. ~W.]{Dept. of Mathematics, Rice University, Houston, TX 77005-1892, USA}
\email{mwolf@rice.edu}

\author{Yunhui Wu}
\address[Y. ~W.]{ Yau Mathematical Sciences Center, Tsinghua University, Beijing 100084, China}
\email{yunhui\_wu@mail.tsinghua.edu.cn}

\date{Compiled \today\ }


\maketitle

\section{Introduction}

The goal of this paper is to prove the following
%a partial result in this direction.
\begin{theorem}\label{thm: main cor}
	Let $M$ be a three-dimensional closed hyperbolic manifold. Then there does not exist a geometric foliation of $M$ by closed minimal surfaces of genus $g>1$.
\end{theorem}

Of course, our first task will be to define the term {\it geometric} in the statement of the result and also to explain the context. We begin with an explanation of the statement: the theorem asserts that such a foliation cannot occur as an instance of a time-dependent geometric flow, in the sense of say, \cite{HuiskenPolden96}.

Indeed, we prove slightly more, in that we do not use the global structure of the fibration. The main theorem is a special case of the following result.

\begin{theorem}\label{thm: main}
	Let $M$ be a three-dimensional hyperbolic manifold. Let $S$ be a closed surface of genus $g>1$ in $M$, and let $N$ be a neighborhood of $S$ in $M$. Then there does not exist a geometric foliation of $N$ by closed minimal surfaces of genus $g>1$.	
\end{theorem}

\begin{remarks} 
\noindent(i) The \mfd \ $M$ in the theorem above does not need to be closed. An example in section~\ref{sec: example} shows that the necessity of the hypothesis that $S$ be closed.

\noindent(ii) The restriction on the genus of the surface $S$ in Theorem~\ref{thm: main} is somewhat superfluous, as minimal surfaces are always saddle-shaped in their ambient spaces: thus, a hyperbolic three-manifold induces on a minimal submanifold a metric of curvature at most $-1$, forcing $S$ to be of hyperbolic type.

\noindent(iii) Because Theorem~\ref{thm: main cor} is an immediate corollary of Theorem~\ref{thm: main}, we see that the exclusion of minimal geometric flows does not depend directly on global dynamical qualities of the flow.
\end{remarks}

Our definition of {\it geometric foliation} relates to the following perspective.  Of course, a foliation $\sF$ of the three-manifold $M$ denotes the decomposition of the manifold $M$ into leaves $F \in \sF$ which are homeomorphic to $S$ with the following property: for every point $p \in M$, there is a neighborhood $U$ of $p$ so that $U$ is covered by the image of a map $F: (-\epsilon, \epsilon)  \times S \to M$, where $F_t=F( \{t\} \times S)$, a leaf of $\sF$, is disjoint from other leaves $F_{t'}$ when $t$ and $t'$ are distinct times.

Because the foliation has leaves of codimension one, it is possible to arrange the mappings $F$ so that the pushforward vectors $\nu = F_*\frac{\partial}{\partial t}$ are normal to the image leaf $F_{t_0}$ for each $t_0 \in (-\epsilon, \epsilon)$.

Then in this setting, a foliation $\sF$ of $M$ is {\it geometric} if the norm $\|\nu\|= \|F_*\frac{\partial}{\partial t}\|$ depends only on the local geometry of the leaf $F_{t_0}=F( \{t_0\}\times S)$, i.e. $\nu$ depends only on the first and second fundamental forms of the leaves of $\sF$. 

Note that in the case where $M$ is hyperbolic and the leaves of $\sF$ are minimal (so that the principal curvatures are additive inverses of one another), the condition that $\nu = \nu(p)$ depends only on the first and second fundamental forms is equivalent to the existence of a function $f = f(\lambda)$ so that $\nu= \nu(p)= \nu(f(\lambda(p)))$ depends only on the size of the principal curvature $\lambda(p)$ of the leaf of $\sF$ through $p$. Thus we may succinctly state the criterion for a foliation to be {\it geometric} as follows.

\begin{Def}\label{c2-fo}
Let $M$ be a three dimensional hyperbolic manifold. We say that \emph{$M$ contains a locally geometric $1$-parameter family of closed minimal surfaces} (or, more briefly, that the minimal foliation is \emph{geometric}) if there exists a closed surface $S$, a constant $\epsilon>0$ and an embedding
\[h:(-\epsilon,\epsilon)\times S \to M\]
such that

\ben
\item the function $h$ is $C^2$ with respect to both $t$ and $p \in S$.

\item for every $t$, each leaf $h_t(\cdot):=h(t, \cdot)\subset M$ is a minimal surface.

\item for any $p \in S$, the function $f(t,p)=<(h_t)_{*}(\frac{\partial}{\partial t}),\vec{n}>|_{t=0}$ only depends on the principal curvature of $S$ at $p$. One may write as $f(0,p)=f(0,||A||^2(p))$ where $||A||^2(p)$ is the square of the second fundamental form of $\{0\} \times S$ at $(0,p)$ in $M$.

\item For time $t=0$, the function $f(0, \cdot): S \to \R$ does not vanish identically.
\een
\end{Def}

In particular, such a foliation would satisfy conditions (1.1) in \cite[Page 45]{HuiskenPolden96} for a time-dependent (i.e. allowed to vary as the leaves vary) geometric flow.

Of course, this definition provides a strong restriction on the possible foliations that are excluded by Theorem~\ref{thm: main cor}. On the other hand, the theorem does rule out foliations defined by local geometric rules, even ones that change from leaf to leaf.  

\subsection{The mathematical and historical context.} Interest in the problem of whether one could possibly foliate a closed hyperbolic three-manifold by minimal surfaces dates back to a paper by Anderson, and in particular to a  conjecture he states \cite[Page 289]{And-83} (see also \cite{Cal-03, Rub-07, Uhl-83} for alternative expressions): 

\begin{conjecture}[Anderson]
If $M$ is a three-dimensional closed hyperbolic manifold, then there does not exist a local $1$-parameter family of closed minimal surfaces in $M$.
\end{conjecture}

Since this conjecture was identified, there have been a few partial results. Both Hass \cite{Has-15} and Huang-Wang \cite{HW-15} have found hyperbolic three-manifolds which fiber over the circle but do not admit any minimal fibration.

The result in Theorem~\ref{thm: main cor} represents something of a different approach to the main problem in that the extra conditions it imposes are on the foliation on any such manifold, rather than on (any foliation on) some particular class of closed three-manifolds.  

\subsection{Method.}

Since the hypothesis we add to the conjecture is a restriction on the foliation, naturally our proof of Theorem~\ref{thm: main} relies on an analysis of the equations governing the geometry of such foliations. We imagine the foliation as determining a flow of minimal surfaces in a hyperbolic three-manifold determined by a function of the local geometry of the minimal surface at a point. Naturally, the geometry of the surface in a hyperbolic three-manifold is determined by its first and second fundamental forms.  Those forms, on any particular minimal leaf, are constrained by Gau\ss's equation and the Simons equations. Most of our interest focuses on the Simons equation on the second fundamental form.   

On the other hand, that the foliation may be construed as a geometric flow provides for a second equation governing the size of the flow vectors.

We then show that these two equations together preclude the existence of the foliation. A brief analysis of this pair of equations results in restrictions on the function $s=\|A\|^2$ (where $A$ is the second fundamental form) and its derivatives  which are not satisfiable on a closed surface. As there are geometric flows on open surfaces (see section~\ref{sec: example}), this last step necessarily uses some topology of closed surfaces: in this case that is some elementary Morse theory on the level sets of the function $s$ in the setting where $s$ is analytic but the Hessian of $s$ does not vanish identically.

\subsection{Organization.} We present the Simons formula in  section~\ref{sec: simons}, and the formula governing the flow in section~\ref{sec: geometric foliation}.   In section~\ref{sec: proof}, we combine these formulas to prove the main result. That section  begins with the governing equations providing a restriction on the geometric function $\phi$: in particular we show that the critical sets for the function $s$ are level sets of $s$. We then conclude with an analysis of how those critical sets might be defined on a closed surface. We close in section~\ref{sec: example} with some examples to show that our restriction of the scope of the theorem to geometric foliations by closed surfaces is necessary.

\subsection{Acknowledgements} The authors appreciate several useful conversations with Zheng (Zeno) Huang on this work.
The first author gratefully acknowledges support from the U.S. National Science Foundation through grant DMS 1564374. He also acknowledges support from U.S. National Science Foundation grants DMS 1107452, 1107263, 1107367 RNMS: Geometric structures And Representation varieties (the GEAR Network). The second author is partially supported by China's Recruitment Program of Global Experts. And he also would like to thank the Department of Mathematics at Rice University where this joint work was partially completed.

%%%%%%%%%%%%%%%%%%%%%%%%%%%%%%%%%%%%%%%%
\section{A Simons Identity}\label{sec: simons}
In this section we apply the Simons identity \cite{Simons-68} to our setting.

Let $M$ be a three-dimensional hyperbolic manifold and $S\subset M$ be an immersed minimal surface. Let $A$ be the second fundamental form of $S$ in $M$ and let $\nabla$ be the covariant derivative with respect to the induced metric on $S$. Let $T(S)$ and $N(S)$ denote the tangent and normal bundles of $S$, respectively; let $\Sym(S)$ denote the bundle of symmetric transformations of $T(S)$, and $H(M)=\Hom(N(S), \Sym(S))$ . We refer to \cite{Simons-68} for the description of some of the objects we use below, in particular, the various operators $\tilde{A} \in \Gamma(\Hom(N(S),N(S))),\underset{\sim}{A} \in \Gamma(\Hom(\Sym(S), \Sym(S)))$, and $B \in \Gamma(\Sym(S)\otimes N(S))$ related to the second fundamental form $A \in \Gamma(H(S))$ used in the next proposition, an adaptation of a computation of Simons \cite{Simons-68}.

\bpro\label{kp}
\[\nabla^2 A=-2A-||A||^2A\]
where $||A||^2$ is the square of the norm of the second fundamental form $A$.
\epro

\bp
Since $M$ is hyperbolic, in particular it is symmetric. So $\overline{R}'$, defined in \cite[Equation (4.2.1)]{Simons-68}, vanishes. The fundamental identity of Simons \cite[Theorem 4.2.1]{Simons-68} is
\be \label{n-1}
\nabla^2 A=-A\circ \tilde{A}- \underset{\sim}{A}\circ A+\overline{R}(A).
\ene
where we will soon recall the definition of $\overline{R}(A)$.

We next apply the argument in the proof of \cite[Theorem 5.3.1]{Simons-68}. Since $S\subset M$ is of codimension $1$, from the definitions of $\tilde{A}$ and $\underset{\sim}{A}$, we obtain that
\be \label{n-2}
\underset{\sim}{A}\circ A=0 \quad \emph{and} \quad A\circ \tilde{A}=||A||^2A.
\ene

The term $\overline{R}(A)$ in \eqref{n-1} is defined by \cite[Equation (4.2.2)]{Simons-68}. We will show below that 
\be \label{n-3}
\overline{R}(A)=-2A.
\ene

Using this formula above, the conclusion then follows from \eqref{n-1}, \eqref{n-2} and \eqref{n-3}. \qed

\emph{Proof of  \eqref{n-3}.} Recall that $M$ has constant curvature $-1$. Hence, for $p \in S$ and $v_1, v_2, v_3 \in T_p(S)$ we have
\be \label{c-e}
\overline{R}_{v_1,v_2}v_3=<v_1,v_3>v_2-<v_2,v_3>v_1.
\ene

Let $e_1,e_2$ be an unit frame in $T_p(S)$, $w$ be the unit normal direction of $S$ in $M$ at $p$ and $B(\cdot, \cdot)$ be defined in \cite[Equation (2.2.2)]{Simons-68}. Pick $x,y\in T_p(S)$. We use \eqref{c-e} to estimate the terms in \cite[Equation (4.2.2)]{Simons-68}. In our setting the dimension of the ambient manifold is $3$ and the submanifold $S$ is of codmension $1$. Now, Simons defines \cite[Equation (4.2.2)]{Simons-68} the operator $\overline{R}(A)$ via its action as
\bear \label{eq-Simons}
\
\\
<\overline{R}^w(A)(x),y>=& \sum_{i=1}^2 (2<\overline{R}_{e_i,y}B(x,e_i),w>+2<\overline{R}_{e_i,x}B(y,e_i),w>\nonumber \\
& -<A^w(x),\overline{R}_{e_i,y}e_i>-<A^w(y),\overline{R}_{e_i,x}e_i> \nonumber \\
&+<\overline{R}_{e_i,B(x,y)}e_i,w>-2<A^w(e_i),\overline{R}_{e_i,x}y>).\nonumber
\eear

Since $B(\cdot,\cdot)$ is orthogonal to $T_p(S)$ by the definition, by \eqref{c-e} we have
\beqar
<\overline{R}_{e_i,y}B(x,e_i),w>&=&0 \\
<\overline{R}_{e_i,x}B(y,e_i),w>&=&0 \\
<A^w(x),\overline{R}_{e_i,y}e_i>&=&<A^w(x),y>-<A^w(x),e_i><y,e_i>\\
<A^w(y),\overline{R}_{e_i,x}e_i>&=&<A^w(y),x>-<A^w(y),e_i><x,e_i>\\
<\overline{R}_{e_i,B(x,y)}e_i,w>&=&<B(x,y),w>=<A^w(x),y>\\
<A^w(e_i),\overline{R}_{e_i,x}y>&=&<A^w(e_i),x><e_i,y>-<A^w(e_i),e_i><x,y>\\
&=&<A^w(x),e_i><e_i,y>.
\eeqar

Substituting the equations above into \eqref{eq-Simons} we obtain 
\beqar
<\overline{R}^w(A)(x),y>&=& \sum_{i=1}^2(-<A^w(x),y>+<A^w(x),e_i><y,e_i>\\
&-&<A^w(y),x>+<A^w(y),e_i><x,e_i>\\
&+&<A^w(x),y>-2<A^w(x),e_i><e_i,y>)\\
&=& -2 <A^w(x),y>\\
&+&\sum_{i=1}^2(<A^w(y),e_i><x,e_i>-<A^w(x),e_i><e_i,y>).
\eeqar

Since
\beqar
\sum_{i=1}^2<A^w(y),e_i><x,e_i>&=&<A^w(y),x> \ \emph{and} \\
 \sum_{i=1}^2<A^w(x),e_i><e_i,y>&=&<A^w(x),y>,
\eeqar

we have
\beqar
<\overline{R}^w(A)(x),y>=-2 <A^w(x),y>.
\eeqar
  
This then proves \eqref{n-3}.
\ep

We interpret Proposition~\ref{kp} into a form that will be more convenient for us.
Let $\Delta$ be the Laplace operator with respect to the induced metric on $S$.
\bt \label{sim-o}
Let $S\subset M$ be an immersed minimal surface where $M$ is three dimensional hyperbolic and $K_S$ be the Gauss curvature of $S$. Then,
away from zeros of $||A||$ we have

\begin{equation}\label{eqn: Laplacian log s}
\Delta \log (||A||^2)=-4 K_S=-2(2+||A||^2).
\end{equation}
\et
\bp
We denote $||A||^2$ by $s$. From the chain rule and Proposition \ref{kp} we have
\bear \label{l-1-1}
\Delta s&=&2<\nabla^2 A,A>+2<\nabla A, \nabla A>\\
&=& -2(2+s)s+2<\nabla A, \nabla A>.\nonumber
\eear

Let $\{e_1,e_2,e_3\}$ be an unit frame at $p \in S$ such that $e_3$ is normal to $S$. Then the second fundamental form $A$ can be written as 
\[A=\sum_{1\leq i, j \leq 2} h_{ij}w_iw_j e_3.\]

Since $S$ is minimal, $h_{11}+h_{22}=0$. Thus,
\[h_{11,k}+h_{22,k}=0.\]

The Gauss-Codazzi equation gives that
\[h_{ij,k}=h_{ik,j}.\]

Thus,
\beqar 
<\nabla A, \nabla A>&=&\sum_{1\leq i,j,k\leq 2}h^2_{ij,k}\\
&=&4h^2_{11,1}+4h^2_{11,2}.
\eeqar

Let $p \in S$ with $s(p)\neq 0$. We may assume that $h_{ij}(p)=\lambda_i \delta_{ij}$ where $\lambda_i$ are principal curvatures. Then
\beqar 
<\nabla s, \nabla s>&=&4\sum_{1 \leq k \leq 2}(\sum_{1\leq i,j\leq 2}h_{ij}h_{ij,k})^2\\
&=&4\sum_{1 \leq k \leq 2}(\lambda_1 h_{11,k}+\lambda_2 h_{22,k})^2.
\eeqar

Since $h_{11,k}+h_{22,k}=0$ and $\lambda_1+\lambda_2=0$, we have
\beqar 
<\nabla s, \nabla s>&=&4\sum_{1 \leq k \leq 2}(\lambda_1-\lambda_2)^2 h^2_{11,k}\\
&=& 8 s \sum_{1 \leq k \leq 2} h^2_{11,k}.
\eeqar

Thus,
\be \label{l-1-2}
<\nabla s, \nabla s>=2s<\nabla A, \nabla A>.
\ene

From \eqref{l-1-1}) and \eqref{l-1-2} we know that away from zeros of $s$,

\bear \label{l-1-3}
\Delta s= -2(2+s)s+\frac{<\nabla s, \nabla s>}{s}.
\eear

Thus, we have that away from zeros of $s$,

\bear \label{l-1-4}
\Delta \log {s}&=&\frac{\Delta s}{s}-\frac{<\nabla s, \nabla s>}{s^2}\\
&=&-2(2+s). \nonumber
\eear

Since $S\subset M$ is minimal, the Gauss equation tells that
\[K_S=-1-s/2.\]

Thus,
\bear \label{l-1-5}
\Delta \log {s}=4K_{S} = -2(2+s).
\eear
\ep

\br \label{remark-1}
If $S$ is closed, the maximum principle together with Theorem \ref{sim-o} yields that the second fundamental form must vanish at some point in $S$. It is well-known \cite{Hopf} that the second fundamental form $A$ can be viewed as the real part of a holomorphic quadratic form on $S$. Thus, $A$ has only finitely many zeros if $S$ is compact.
\er

%%%%%%%%%%%%%%%%%%%%%%%%%%%%%%%%%%%%%%%

\section{An equation for a minimal foliation}\label{sec: geometric foliation}

In addition to equation \eqref{eqn: Laplacian log s}, we will need an equation governing the size of the flow vector: deriving that relationship is the goal of this section.

Let $M$ be a three-dimensional hyperbolic manifold. Assume that there exists a local one-parameter family of minimal surfaces in $M$. More precisely, let $\epsilon>0$, let $S$ be a surface and suppose there exists a differentiable embedding 
\[h:(-\epsilon,\epsilon)\times S \to M\]
such that for every $t$, each leaf $h_t(\cdot):=h(t, \cdot)\subset M$ is a (distinct) minimal surface.

Denote $h(0,S)$ by $S$ for simplicity. Let $\vec{n}$ be the unit normal vector field on $S$. Then there exists a positive function $f\in C^{2}(S)$ such that 
\begin{equation} \label{eqn: laplacian f}
((h_0)_{*}(\frac{\partial}{\partial t}))^{\perp}=f\cdot \vec{n}.	
\end{equation}  
where we have indicated by $\perp$ the projection to the normal bundle to the leaf.

\bpro \label{pr-m}
\[\Delta f=(2-||A||^2)f.\]
\epro

\bp
We use the same notations as in \cite{Simons-68}.

It follows from \cite[Theorem 3.3.1]{Simons-68} that $f\cdot \vec{n}$ is a Jacobi field. That is, 
\be \label{l-2-0}
\nabla ^2 (f\cdot \vec{n})=\overline{R}(f\cdot \vec{n})-\tilde{A}(f\cdot \vec{n}).
\ene

We next use that $M$ has constant curvature $-1$. Hence, for $p \in S$ and $v_1, v_2, v_3 \in T_p(S)$ we have
\be \label{c-e-2}
\overline{R}_{v_1,v_2}v_3=<v_1,v_3>v_2-<v_2,v_3>v_1.
\ene

Let $e_1,e_2$ be an unit frame in $T_p(S)$. It follows from \cite[Equation 3.2.1]{Simons-68} and \eqref{c-e-2} that 

\bear \label{l-2-1}
\overline{R}(f\cdot \vec{n})&=&\sum_{i=1}^2 (\overline{R}_{e_i, f\vec{n}}e_i)^{\perp}\\
&=& \sum_{i=1}^2(f\vec{n}-<e_i,f\vec{n}>e_i)^{\perp} \nonumber \\
&=&2f\cdot \vec{n}. \nonumber
\eear

The term $\tilde{A}$ in \eqref{l-2-1} is defined in \cite[Equation 2.2.5]{Simons-68}. It follows from \cite[Equation 2.2.7]{Simons-68} that 
\bear \label{l-2-2}
\tilde{A}(f\cdot \vec{n})&=&<\tilde{A}(f\cdot \vec{n}),\vec{n}>\vec{n} \\
&=& <\tilde{A}( \vec{n}),\vec{n}>f\cdot \vec{n} \nonumber \\
&=& ||A||^2 f\cdot \vec{n}\nonumber
\eear

It follows from \eqref{l-2-0}, \eqref{l-2-1} and \eqref{l-2-2} that 
\be \label{l-2-4}
\nabla ^2 (f\cdot \vec{n})=(2-||A||^2)(f\cdot \vec{n}).
\ene

On the other hand, after extending $e_1,e_2,\vec{n}$ to vector fields $E_1,E_2, N$ such that they are pairwise orthogonal and 
$\nabla_{E_{i}}E_j(p)=0$ and $\nabla_{E_{i}}N(p)=0$, it then follows from \cite[Proposition 1.2.1]{Simons-68} that, evaluated at $p$, we have
\bear  \label{l-2-5}
\nabla ^2 (f\cdot \vec{n})&=&\sum_{i=1}^2 \nabla_{E_i}\nabla_{E_i}(f\vec{n}) \\
&=& \Delta(f) \cdot \vec{n}+f \cdot \nabla_{E_i}\nabla_{E_i}(\vec{n}) \nonumber \\
&=& \Delta(f) \cdot \vec{n}. \nonumber
\eear 
In the last equality above we apply that at $p$, $$<\nabla_{E_i}\nabla_{E_i}(\vec{n}),\vec{n}>=-<\nabla_{E_i}(\vec{n}),\nabla_{E_i}(\vec{n})>=0.$$

Thus, it follows from \eqref{l-2-4} and \eqref{l-2-5} that
\be \notag
\Delta f=(2-||A||^2) f
\ene
as desired.
\ep
%%%%%%%%%%%%%%%%%%%%%%%%%%%%%%%%

\section{Proof of Theorem \ref{thm: main}} \label{sec: proof}
In this section we will finish the proof of Theorem \ref{thm: main}. We use the same notations as in the previous sections. 

Let $M$ be a three-dimensional hyperbolic manifold and $S$ be a closed surface. Assume that  
\[h:(-\epsilon,\epsilon)\times S \to M\]
is a local $C^2$ family of minimal surfaces in $M$ which is geometric. That is, 
\ben
\item $h$ is $C^2$ with respect to both $t$ and $p$.

\item $h$ is an embedding.

\item for every $t$, each leaf $h_t(\cdot):=h(t, \cdot)\subset M$ is a minimal surface.

\item for any $p \in S$, the function $f(t,p)=<(h_t)_{*}(\frac{\partial}{\partial t}),\vec{n}>|_{t=0}$ only depends on the principal curvature of $S$ at $p$. One may write as $f(0,p)=f(0,s(p))$ where $s(p)=||A||^2(p)$.

\item For time $t=0$, the function $f(0, \cdot): S \to \R$ does not vanish identically.
\een

Recall that $S=h_0(S)$ and $\Delta$ is the Laplace operator with respect to the induced metric on $S$. Theorem \ref{sim-o} and Proposition \ref{pr-m} then assert that the following system of partial differential equations applies to a geometric foliation of minimal surfaces in a hyperbolic three-manifold:
\bear \label{sys-1}
\begin{cases}
\Delta \log {s}=-2(2+s)\\
\Delta f=(2-s)f.
\end{cases}
\eear

We will now show that this system admits no solutions under our assumptions on the local structure of this hyperbolic manifold near a leaf of the foliation.

First the chain rule gives that 
\bear \label{l-3-1}
\Delta f(0,p)&=&\Delta f(0,s(p))\\
&=& \frac{\partial^2}{\partial s^2}f(0,s) \cdot ||\nabla s||^2+ \frac{\partial}{\partial s}f(0,s) \cdot \Delta s. \nonumber
\eear

So we have
\bear \label{l-3-2}
 \frac{\partial^2}{\partial s^2}f(0,s) \cdot ||\nabla s||^2+ \frac{\partial}{\partial s}f(0,s) \cdot \Delta s=(2-s)f(0,s).
\eear

Recall that Theorem \ref{sim-o} gives that
\bear \label{l-3-3}
s \cdot (\Delta s)-||\nabla s||^2=-2s^2(2+s).
\eear

Eliminating $\Delta s$, we obtain
\bear \label{l-3-4}
 ||\nabla s||^2=\frac{s(2-s)f(0,s)+2s^2(2+s)\cdot \frac{\partial}{\partial s}f(0,s)}{s\cdot \frac{\partial^2}{\partial s^2}f(0,s)+ \frac{\partial}{\partial s}f(0,s)}
\eear
at $(t,s)$ such that $s\cdot \frac{\partial^2}{\partial s^2}f(t,s)+ \frac{\partial}{\partial s}f(t,s)\neq 0$.  We will refine this analysis in the next lemma.

To that end, define $\mathcal{C}:=\{p \in S; \ \nabla s(p)=0\}$ which is the set of critical points of $s$ in $S$.
A direct consequence of \eqref{l-3-4} is 
\bl \label{l-level}
The set $\mathcal{C}$ consists of level subsets of $s:S\to \R^{\geq 0}$. More precisely, assume that $p \in \mathcal{C}$, then for any $q\in S$ with value $s(q)=s(p)$ we have
\[q \in \mathcal{C}.\]
\el

\bp
We begin with the equation $\Delta f=(2-s)f$ from \eqref{sys-1}. First $s$ is analytic on $S$ because the second fundamental form $A$ can be viewed as the real part of a holomorphic quadratic form on $S$ \cite{Hopf}. By classical Schauder theory for elliptic partial differential equations (see \cite[Page 110]{GT-pde} for details) we know that since $f$ satisfies the elliptic PDE \eqref{sys-1}, the solution $f$ is also analytic. 

Set $$\phi_1(s):=s(2-s)f(0,s)+2s^2(2+s)\cdot \frac{\partial}{\partial s}f(0,s)$$
and 
$$\phi_2(s):=s\cdot \frac{\partial^2}{\partial s^2}f(0,s)+ \frac{\partial}{\partial s}f(0,s),$$
so that the real-valued functions $\phi_1(s)$ and $\phi_2(s)$ of $s$ are the numerator and denominator of \eqref{l-3-4}.

Let $p \in \mathcal{C}$. From \eqref{l-3-4} we know that $\phi_1(s(p))=0$.

Case-1. If $\phi_2(s(p))\neq 0$, we are done, as \eqref{l-3-4} displays $||\nabla s||^2$ as a function of only $s$ (and $f(s)$): all points taking on a critical value of $s$ are critical for $s$.

Case-2. If $\phi_2(s(p))=0$, since both $f$ and $s$ are analytic, and since $f$ does not vanish identically (see Definition~\ref{c2-fo}(iv)),  the Taylor expansions at $s(p)$ may be written as
\bear \label{1-equ}
\phi_1(s)=\sum_{k\geq n_1}a_k (s-s(p))^k \ \emph{where $a_{n_1}\neq 0$, for some $n_1 \in \N$}
\eear
and 
\bear \label{2-equ}
\phi_2(s)=\sum_{k\geq n_2}b_k (s-s(p))^k \ \emph{where $b_{n_2}\neq 0$ for some $n_2 \in \N$}.
\eear

It is clear that $||\nabla s||^2$ is smooth on the minimal surface $S$, hence so is $||\nabla s||^2=\frac{\phi_1(s)}{\phi_2(s)}$. In particular, we have
$$n_1 \geq n_2.$$ 

Thus, from \eqref{l-3-4}, \eqref{1-equ} and \eqref{2-equ} we now see that for any $q\in S$ with $s(q)=s(p)$,
\bear
||\nabla s(q)||^2=\frac{a_{n_1}}{b_{n_2}} \ \emph{if $n_1=n_2$, and } ||\nabla s(q)||^2=0 \ \emph{if $n_1 > n_2$}.
\eear

That is, the set of critical points of $s$ is a level subset of $s$. 
\ep

\bl \label{l-maximal}
If $\mathcal{C}$ contains a smooth arc $c$, then for any $p \in c$,
\[s(p)=\max_{q\in S}s(q).\]
\el
\bp 
First by Remark (\ref{remark-1}) we know that $s|_c \neq 0$ since $c$, as an arc, contains a continuum, while the the set $s^{-1}(0)$ is the zero set of the holomorphic quadratic differential on $S$ defined as the complexification of $A$, whose zero set is discrete. Let $p \in c$, $X \in T_p(c)$ and $Y\in T_p(S)$ such that $\{X,Y\}$ extends to a unit frame defined near $p$ in $T_p(S)$ and the vector field $X$ is tangent to that arc $c$. At $p$ we have  
\bear \label{l-3-5}
\Delta s=X X (s)-(\nabla_XX)(s)+Y Y (s) -(\nabla_YY)(s).
\eear
Since $p \in c \subset \mathcal{C}$, we have that since $(\nabla_XX)$ is orthogonal to the arc $c$, and $c$ is critical for $s$, then we must have $(\nabla_XX)(s) (p)=0$. Similarly $(\nabla_YY)(s)=0$, again as $p$ is a critical point for $s$. Lemma \ref{l-level} then gives that $s|_{c}\equiv s(p)$. So $X X (s) (p)=0$. Thus,
\bear \label{l-3-6}
\Delta s(p)=YYs(p).
\eear

We thus conclude from \eqref{l-3-3}, that for $p \in \mathcal{C}$, we have
\bear \label{l-3-7}
YYs(p)= \Delta s(p) =-2s(p)(2+s(p))<0.
\eear

Now, by definition the arc $c \subset \mathcal{C}$ consists of critical points, so since the field $X$ is tangent to that arc $c$, we have 
\bear \label{l-3-8}
XXs(p)=0, Xs(p) = 0 \quad \emph{and} \quad Ys(p)=0.
\eear

It then follows from \eqref{l-3-7} and \eqref{l-3-8} that for any $p \in \mathcal{C}$, the value $s(p)$ is a local maximum. 

Moreover, by Lemma \ref{l-level}, for any $p \in \mathcal{C}$ the value $s(p)$ is also a global maximum: to see this, connect $p$ to a global maximum by a path, say $\Gamma$, whose initial point is at $p$ and whose terminal point is the global maximum.  Then along the path $\Gamma$, because $p$ is a local maximum for $s$, the value of $s$ first declines then attempts to rise to the value for the global maximum: the intermediate value theorem then provides for a later first $q \in \Gamma$ for which $s(q) = s(p)$.  But at that level $s(q)$, we have from \eqref{l-3-4} that $q$ is again a critical point for $s$.  If that point $q \in \Gamma$ is a saddle point, then the level set of $s$ through $q$ locally separates values of $s$ larger than $s(q)$ from those smaller than $s(q)$ and hence contains an arc. Hence the level set of $s$ through $q$ contains an arc and $q$ is a local maximum by the argument above. Iterating this argument yields that the maximum that $s$ can achieve on $\Gamma$ is actually the value $s(p)$, as claimed.
The proof is complete.
\ep

\begin{cor}\label{critical}
If $p \in \mathcal{C}$, then either $s(p)=0$ or $s(p)=\max_{q\in S}s(q).$
\end{cor}

\begin{proof}
First, the function $s$ is real-analytic on $S$ because the second fundamental form $A$ can be viewed as the real part of a holomorphic quadratic form on $S$ \cite{Hopf}.
Secondly \eqref{l-3-7} shows that at non-zero critical points, we have that the Hessian $\Hess s$ does not vanish identically. Hence, any arc in a level subset of $s$ is smooth (one may see \cite{AP} for more details). As noted in the proof of Lemma~\ref{l-maximal}, if $p \in \mathcal{C}$, then if $p$ is a saddle point, then the $s(p)$-level set of $s$ must contain a smooth arc, and hence $s$ attains its global maximum at $p$.  Since $s \geq 0$ but has zeroes at only the (finitely many) zeroes of $A$ (see Remark (\ref{remark-1})), we see that the only critical values obtainable are either global maxima or zeroes (global minima): these account for all the critical points in $\mathcal{C}$.
\end{proof}

We are now ready to prove Theorem \ref{thm: main}.
\bp[Proof of Theorem \ref{thm: main}]
We have not established that $s: S \to \R$ is a global Morse function, so we cannot immediately apply  Morse theory to conclude the argument.  Our argument is hence just a bit more involved.

First since $S$ is a minimal surface in a three-dimensional hyperbolic space, the Gauss equation gives that the Gauss curvature $K_S$ of $S$ is less than $-1$. In particular $S$ is a closed surface of genus $g\geq 2$, which is not simply connected. 

On the other hand, let $m=\max_{q\in S}s(q)$. Since $s$ is analytic on $S$ (note once again that the second fundamental form $A$ can be viewed as the real part of a holomorphic quadratic form on $S$ \cite{Hopf}), and because \eqref{l-3-7} shows that at critical points we have that $\Hess s$ does not vanish identically, it follows (see also \cite[Lemma 3]{AP}) that the level set $s^{-1}(m)$ consists only of a finite number of isolated points and a finite number of circles. Thus, one may choose a \nbhd \ $V_1$ of $s^{-1}(m)$ such that $V_1$ is a collection of disks and annuli. Set
\[V_2=S \setminus \{s^{-1}(m)\}.\]
From Corollary \ref{critical}, the only critical points of $s$ on $V_2$ are (finite) zeroes (absolute minima), so it follows from the standard Morse theory \cite{Milnor-M} 
that $V_2$ is topologically trivial. That is, the open set $V_2$ is homeomorphic to a two-dimensional disk. In particular, the Euler characteristic $\chi(V_2)=1$. Since $S=V_1 \cup V_2$, we have that the Euler characteristic $\chi(S)$ may be estimated by

\begin{align*}
\chi(S) &= \chi(V_1) + \chi (V_2)-\chi(V_1 \cap V_2)\\
&= \chi(V_1) + \chi (V_2) \\
&\geq 0 +1\\
&=1.
\end{align*}
Here the second equality follows from properties of the Euler characteristic $\chi$ when one decomposes a surface into subsurfaces and that the intersection $V_1 \cap V_2$ is homotopic to a collection of circles, each of which contributes zero to the sum. The inequality follows because $V_1$ is a collection of finite disks and finite annuli. We conclude that the orientable surface $S$ must have genus zero, contradicting the conclusion of our first paragraph.
\ep

%%%%%%%%%%%%%%%%%%%%%%%

\section{A nontrivial example for minimal disk foliation in $\H ^3$} \label{sec: example}
The argument just above finishing the proof of Theorem~\ref{thm: main} could be construed to leave open the possibility of a geometric foliation by minimal (topological) punctured spheres (i.e. disks).  In this concluding section, we exhibit a not-quite-trivial family, suggesting a sharpness to our result.

We begin by noting the trivial example: consider $\H^3$ as the upper half-space with coordinate $(x,y,z)$ endowed with the standard hyperbolic metric $ds^2=\frac{dx^2+dy^2+dz^2}{z^2}$. It is clear that the family $\{(t, y, z); y \in \R, z \in \R^{>0}\}_{t\in (\frac{-1}{2},\frac{1}{2})}$ is a foliation by minimal (actually totally geodesic) surfaces. 

Leaving this trivial example aside, we remark in the remainder of this section on a different minimal foliation whose leaves are not totally geodesic. We use the same notations as in \cite{Kok-97}, whose example  Example 7.2 in \cite{Kok-97} we adapt for our purpose.

Consider the three-dimensional hyperbolic space $(\R^3, ds^2)$ with Fermi coordinates $(t,x,y)$ where
\[ds^2=dt^2+e^{-2t}(dx^2+dy^2).\]

Define
\beqar
f: \R^2 \times (0,1) &\to& (\R^3, ds^2) \\
   ((u,v),t) &\mapsto& (\rho(u),t \cdot \int e^{2\rho(u)}du,v)
\eeqar
where $\rho(u)$ solves the ODE  $$\frac{d\rho}{du}=(e^{-2\rho}-t^2 e^{2\rho})^{\frac{1}{2}}.$$

Kokubu \cite[Page 377]{Kok-97} shows that for each $t$, the image $\{f((\cdot, \cdot), t)\}$, denoted by $\Sigma_t$, is a minimal surface in $(\R^3, ds^2)$. Thus, the family $\Sigma_t$ is a minimal foliation. We will show that this minimal foliation is geometric and none of the leaves is totally geodesic.\\

Fix $t$; then a direct computation gives that
\beqar
\frac{\partial f}{\partial u}=(\frac{d\rho}{du}, t e^{2\rho(u)}, 0) \quad \emph{and} \quad
\frac{\partial f}{\partial v}=(0, 0, 1).
\eeqar

Then,
\beqar
<\frac{\partial f}{\partial u},\frac{\partial f}{\partial u}>&=&(\frac{d\rho}{du})^2+e^{-2\rho(u)}\cdot (t e^{2\rho(u)})^2=e^{-2\rho(u)}\\
<\frac{\partial f}{\partial u},\frac{\partial f}{\partial v}>&=&0 \\ 
<\frac{\partial f}{\partial v},\frac{\partial f}{\partial v}>&=& e^{-2 \rho(u)}.
\eeqar

Thus, the induced metric $ds^2_{\Sigma_t}$ on $\Sigma_t$ is 
\[ds^2_{\Sigma_t}=e^{-2\rho(u)}\cdot (du^2+dv^2), \]

and the unit normal vector $\vec{n}$ of $ds^2_{\Sigma_t}$ is 
\[\vec{n}=\frac{(-t,\frac{d\rho}{du},0)}{\sqrt{t^2+e^{-2\rho(u)}(\frac{d\rho}{du})^2}}=e^{2\rho(u)}\cdot (-t,\frac{d\rho}{du},0).\]

A direct computation gives that the Gauss curvature $K(\Sigma_t)$ of $ds^2_{\Sigma_t}$ is 
\[K(\Sigma_t)=-1-t^2 e^{2\rho(u)}.\]
Since $t\in (0,1)$, we have $\Sigma_t$ is not totally geodesic in $(\R^3, ds^2)$.\\

As usual, let $s=|A|^2$ be square of the norm of the second fundamental form $A$ of $\Sigma_t$ in $(\R^3, ds^2)$. The Gauss equation gives that 
\beqar
s=|A|^2=2\cdot (-1-K(\Sigma_t))=2t^2 e^{2\rho(u)}.
\eeqar

The derivative of  $\Sigma_t$ in the $t$-direction is 
\[\frac{\partial f}{\partial t}=(0, \int e^{2\rho(u)}du, 0).\]

Then,
\[<\frac{\partial f}{\partial t},\vec{n}>=e^{-2\rho}\int e^{2\rho(u)}du \cdot e^{2\rho(u)}\frac{d\rho}{du}=\int e^{2\rho(u)}du \cdot(e^{-2\rho}-t^2 e^{2\rho})^{\frac{1}{2}},\]
which is denoted by $F(u,t)$.

Since $\rho$ is increasing with respect to $u$ and $s=2\cdot (-1-K(\Sigma_t))=2t^2 e^{2\rho(u)}$, we may also write $F(u,t)$ as $F(s,t)$ which is a function only depending on $s$ and $t$. Hence,

\[(\frac{\partial \Sigma_t}{\partial t})^{\perp}=F(s,t) \cdot \vec{n}.\] 

Therefore, the family $\Sigma_t$ is a geometric minimal foliation whose leaves are not totally geodesic.

%%%%%%%%%%%%%%%%%%%%%%%%%%%%%
\bibliographystyle{amsalpha}
\bibliography{wp}
\end{document}